\documentclass[11pt,twoside]{amsart}

\usepackage[utf8]  {inputenc}%
\usepackage[T1]      {fontenc }%
\usepackage          {amsmath }%
\usepackage          {amsfonts}%
\usepackage          {amssymb }%
\usepackage          {amsthm  }%
\usepackage          {a4wide  }%
\usepackage          {url     }%
\usepackage          {tikz    }%
\usepackage[bookmarks=false,pdfborder={0 0 0.05}]{hyperref}
\usepackage[all]{xy}
\usepackage{lmodern}
\usepackage{tikz-cd}

\usepackage{enumitem}

\usepackage{booktabs}

\usepackage{amsmath, amsfonts, amssymb, amsthm,  wasysym, graphics, graphicx, xcolor, frcursive,comment,bbm}

\definecolor{darkblue}{rgb}{0.0,0,0.7} 
\definecolor{darkred}{rgb}{0.7,0,0} 

\usepackage{hyperref}
\usepackage[all]{xy}
\usepackage[T1]{fontenc}

\usepackage{MnSymbol}

\DeclareMathOperator{\idop}{id}

\newtheorem{theorem}{Theorem}[section]

\newtheorem{proposition}[theorem]{Proposition}
\newtheorem{lemma}[theorem]{Lemma}

\theoremstyle{definition}
\newtheorem{definition}[theorem]{Definition}

\newtheorem{remark}[theorem]{Remark}

\newtheorem*{theoremA}{Theorem A}
\newtheorem*{theoremB}{Theorem B}

\author[P.~Bader]{Philipp Bader}
\address{Philipp Bader, University of Glasgow, United Kingdom}
\email{p.bader.1@research.gla.ac.uk}

\author[B.~Böhmler]{Bernhard Böhmler}
\address{Bernhard Böhmler, Leibniz Universität Hannover, Germany}
\email{boehmler@math.uni-hannover.de}

\author[P.~Wegener]{Patrick Wegener}
\address{Patrick Wegener, Philipps-Universit\"at Marburg, Germany}
\email{wegenerp@staff.uni-marburg.de}

\title[commutators of signed n-cycles]{commutators of signed n-cycles}

\date{\today}



\begin{document}
\begin{abstract}
    We show that for $n \geq 6$ each element of the commutator subgroup in the symmetric group~$\mathfrak{S}_n$ resp. in the signed symmetric group $(\mathbb{Z}/2\mathbb{Z})^n\rtimes\mathfrak{S}_n$ is the commutator of two $n$-cycles resp. the commutator of two $n$-cycles with a negative sign product; with one exception. If $n \equiv 2$ $\mathrm{mod}~4$, the element $-\mathrm{id}$ of $(\mathbb{Z}/2\mathbb{Z})^n\rtimes\mathfrak{S}_n$ is not such a commutator. In the language of Coxeter groups, this yields a description of commutators of Coxeter elements in Type $A$ and $B$.  
\end{abstract}

\maketitle

\section{Introduction}
The commutator subgroup of a group is defined as the subgroup generated by all commutators. A product of two commutators need not be a commutator itself. Hence, it makes sense to ask whether the commutator subgroup of a given group consists entirely of commutators.

In 1951, Ore proved that the above holds for the finite symmetric groups by showing that every element of the alternating group can be written as a commutator in the symmetric group \cite{Ore}. Moreover, Ore proceeded to show that for $n \geq 5$, every element in the alternating group $\mathfrak{A}_n$ is already a commutator in $\mathfrak{A}_n$. This led to the so-called \textit{Ore conjecture} which asserts that every element in a finite, non-abelian, simple group is a commutator. Over the years the conjecture was established for more simple groups and finally completely proven in~2010 by Liebeck, O’Brien, Shalev and Tiep \cite{Liebeck}.

We show the following results.
\begin{theoremA}[Theorem~\ref{thm:type-A}]
Let $n \geq 6$ and let $C(n)$ be the conjugacy class of $n$-cycles in the symmetric group $\mathfrak{S}_n$. Then the map
$$C(n) \times C(n) \to \mathfrak{A}_n,\quad (\tau, \pi) \mapsto [\tau, \pi] = \tau^{-1} \pi^{-1} \tau \pi$$
is surjective. 
\end{theoremA}
Note that this result can be derived from the work of Vavpeti\v{c} \cite{vavpetic}. Nevertheless, we will provide a proof here since our proof strategy will be essential for the proof of our second result:

\begin{theoremB}[Theorem~\ref{thm:type-B}]
Let $n \geq 6$ and let $C(n)^-$ be the conjugacy class of $n$-cycles with a negative sign product in the signed symmetric group $W_{B_n} :=(\mathbb{Z}/2\mathbb{Z})^n\rtimes\mathfrak{S}_n$. Then the map
\begin{align*}
C(n)^- \times C(n)^- \to & \begin{cases}
   \left[W_{B_n}, W_{B_n} \right]\setminus\{-\mathrm{id}\}, & n\equiv 2 \pmod 4\\
   \left[W_{B_n}, W_{B_n} \right], & \text{else}\\
\end{cases} \\
(\tau, \pi) \mapsto & [\tau, \pi]
\end{align*}
is surjective. 
\end{theoremB}

Note that, in the language of Coxeter groups, both theorems describe precisely the set of commutators $[c,d]$, where $c$ and $d$ are Coxeter elements in a Coxeter group of type $A$ or $B$.

\medskip
\textbf{Outline of the paper.} In Section~\ref{sec:notation} we introduce the notation which we will use throughout the paper. In Section~\ref{sec:A} we prove Theorem A. In Section~\ref{section:type-B} we introduce the signed permutation groups and collect some results about these groups. In Section~\ref{sec:5} we introduce balanced bipartitions and investigate which vectors in $\mathbb{F}_2^n$ can be written as sums of permuted vectors. Finally in Section~\ref{sec:ProofThmB} we prove Theorem B.

\medskip
\textbf{Disclosure of AI tools.} Generative AI tools, namely ChatGPT 5.6, were used to prove Lemma~\ref{lem:ImageOfMpi} and for consistency checks. All AI-generated suggestions were independently reviewed and verified by the authors, who assume full responsibility for the mathematical arguments, results, and final content of the paper.

\section{Notation and convention} \label{sec:notation}
\indent Let $n$ be a positive integer. We denote the symmetric group acting on $n$ letters by $\mathfrak{S}_n$ and the alternating group acting on $n$ letters by $\mathfrak{A}_n$.\\
\indent Every permutation $\pi \in \mathfrak{S}_n$ may be uniquely expressed (up to the order of factors) as a product $\pi = \tau_1 \dots \tau_r$ of disjoint cycles. Let $\tau_1, \ldots, \tau_r$ be arranged in increasing order of cycle length and let $\tau_i$ be a cycle of length $k_i$. We refer to $(k_1,\ldots, k_r)$, considered up to permutation of its entries, as the \textbf{cycle type} of $\pi$, and denote the corresponding conjugacy class in $\mathfrak{S}_n$ by~$C(k_1,\ldots,k_r)$. Here, the order of the $k_i$ will always be chosen for convenience.\\
\indent If $x,y$ are elements of a group, we write
$$[x,y] = x^{-1}y^{-1}xy$$
for the commutator of $x$ and $y$. Accordingly, we read permutations from right to left:\linebreak if~$\pi,\tau\in\mathfrak{S}_n$ then $ (\pi\tau)(i)=\pi(\tau(i))$ for $1\leq i\leq n$.

\section{The Symmetric Group} \label{sec:A}

In this section, we consider the Coxeter group of type $A_n$, which is isomorphic to the symmetric group $\mathfrak{S}_{n+1}$. 
The Coxeter elements of $W_{A_n} \cong \mathfrak{S}_{n+1}$ are exactly the elements of the conjugacy class $C(n+1)$. Our main result specifies which elements arise as the commutator of two Coxeter elements.

We note that the result presented in this section largely follows from the work of Vavpeti{\v{c}} \cite{vavpetic}. The proof we present is conceptually no different from that in \cite{vavpetic}. The reason we include it here is that in our proof, we construct permutations more explicitly, which will be useful in Section \ref{sec:ProofThmB} when dealing with signed permutations. We make explicit citations to the article \cite{vavpetic} throughout and highlight the analogy of the results. 

\begin{theorem}\label{thm:type-A}
    Let $n \geq 6$. The map
    $$C(n) \times C(n) \to \mathfrak{A}_n, (\tau, \pi) \mapsto [\tau, \pi]$$
    is surjective. 
\end{theorem}

In Coxeter language, the theorem states that for $n \geq 5$, every element in the commutator subgroup of $W_{A_n}$ can be expressed as the commutator of two Coxeter elements. 

\begin{remark}
    For $W_{A_1}$ and $W_{A_2}$, only the identity element can be expressed as a commutator of Coxeter elements. For $W_{A_3}$ and $W_{A_4}$, both the identity and all elements in the conjugacy classes $C(3,1)$, $C(3,1,1)$, and $C(5)$, respectively, can be expressed as commutators of Coxeter elements. 
\end{remark}

The rest of this section is devoted to the proof of the theorem. For $n \in \mathbb{N}$, let 
$$\sigma_n := (1 \ 2 \ \ldots \ n) \in C(n).$$ 

\begin{definition}
    We call a normalised $n$-cycle $(1\ 2\ i_3 \ \ldots \ i_n)$ \textbf{parity stable} if $i_j \equiv j \pmod 2$ for $3 \leq j \leq n$ and \textbf{parity unstable} otherwise. If a normalised $n$-cycle $(1, 2, i_3 \ldots, i_n)$ is parity unstable, we call every $k\in \{3,\ldots, n\}$ with $i_k \not\equiv k \pmod 2$ a \textbf{parity defect} of the cycle.
\end{definition}

Note that $\sigma_n$ is parity stable. We start by showing that it suffices to prove the following proposition.

\begin{proposition}\label{prop:type-A}
    For every $n \geqslant 6$ and every non-trivial even conjugacy class $C(k_1, \ldots , k_m)$ in $\mathfrak{S}_n$, there exists a parity unstable $\pi \in C(n)$ with $[\sigma_n, \pi] \in C(k_1, \ldots , k_m).$
\end{proposition}

\begin{proof}[Proof of Theorem \ref{thm:type-A} given Proposition \ref{prop:type-A}]
    Let $n \geq 6$ and $a \in \mathfrak{A}_n$. For $a = \mathrm{id}$ take $x = y = \sigma_n$. Assume from now on that $a \neq \mathrm{id}$. By Proposition \ref{prop:type-A}, we can find an element $a'$ contained in the same $\mathfrak{S}_n$-conjugacy class as $a$ such that there is $\pi \in C(n)$ with $[\sigma_n, \pi] = a'$. Furthermore, since $a$ and $a'$ are conjugate, there is $\tau \in \mathfrak{S}_n$ with $\tau a' \tau^{-1} = a$. Let $x := \tau \sigma_n \tau^{-1}$ and $y := \tau \pi \tau^{-1}$. It follows that $x, y \in C(n)$ and $[x, y] = a$, thereby concluding the proof.
\end{proof}

\begin{remark}
    The fact that the permutation $\pi$ in Proposition \ref{prop:type-A} is parity unstable is not necessary for the proof of Theorem \ref{thm:type-A}. However, we will make use of this fact in Section \ref{sec:ProofThmB} which is why it is included in the statement of the proposition.
\end{remark}
 
To prove Proposition \ref{prop:type-A}, we will need Lemma \ref{lem:Bader1} and Lemma \ref{lem:Bader2} below. These two results can be deduced from \cite[Lemma 2.9]{vavpetic}. We include short proofs which in particular contain the construction of the permutations $\pi_+$ and $\pi$ appearing in the lemmata, since we will need to determine their parity stability. In the proofs, for two permutations $\alpha$ and $\beta$ on different elements, we use the notation $\alpha \sqcup \beta$ to indicate the permutation that does both at the same time. For example, if $\alpha = (1 \, 2 \, 3)$ and $\beta = (4 \, 5) \, (6),$ then $\alpha \sqcup \beta = (1 \, 2 \, 3) \, (4 \, 5) \, (6).$

\begin{lemma}\label{lem:Bader1}
Let $n\geq 2$ and $\pi\in C(n)$ with $\pi(1)=2$ such that $[\sigma_n,\pi]\in C(k_1, \ldots,k_m)$. Then there exists $\pi_{+}\in C(n+1)$ with $\pi_{+}(1)=2$ such that $[\sigma_{n+1},\pi_{+}]\in C(k_1,\ldots , k_m,1)$.
\end{lemma}
\begin{proof}
Since $\pi$ is an $n$-cycle with $\pi(1) = 2,$ we can write $\pi = (1 \ 2 \ i_3 \ \ldots \ i_n)$ where $i_j \in \{3, \ldots , n\}$. Define $\pi_+ := (1 \ \ 2 \ \ i_3 \ \ \ldots \ \ i_n \ \ n+1) \in C(n+1)$. Let $\alpha := [\sigma_n, \pi] \in C(k_1, \ldots , k_m).$ Then, $[\sigma_{n+1}, \pi_+] = \alpha \sqcup (n+1) \in C(k_1, \ldots , k_m, 1).$ 
\end{proof}

\begin{lemma}\label{lem:Bader2}
Let $n_1, n_2 \geq 2$ and $n := n_1+n_2$. Let $\pi_1 \in C(n_1), \pi_2 \in C(n_2)$ with $\pi_1(1) = \pi_2(1) = 2$ such that $[\sigma_{n_1}, \pi_1] \in C(k_1, \ldots , k_m)$ and $[\sigma_{n_2}, \pi_2] \in C(\ell_1, \ldots , \ell_s)$. Then there exists $\pi \in C(n)$ with $\pi(1) = 2$ such that
$$
[\sigma_n, \pi] \in C(k_1, \ldots, k_m, \ell_1, \ldots , \ell_s).
$$
\end{lemma}
\begin{proof}
Analogously to the proof of the previous lemma, we write $\pi_1 = (1 \ 2 \ i_3 \ \ldots \ i_{n_1}),$ $\pi_2 = (1 \ 2 \ j_3 \ \ldots \ j_{n_2}),$ where $i_l \in \{3, \ldots , n_1\}$ and $j_m \in \{3, \ldots , n_2\}.$ Let $\alpha := [\sigma_{n_1}, \pi_1] \in C(k_1, \ldots , k_m), \ \beta := [\sigma_{n_2}, \pi_2] \in C(\ell_1, \ldots , \ell_s)$ and define 
$$\pi := (1 \ \ 2 \ \ i_3 \ \ ... \ \ i_{n_1} \ \ n_1+1 \ \ n_1+2 \ \ n_1+j_3 \ \ ... \ \ n_1+j_{n_2}).$$ 
By definition $\pi \in C(n),$ and it is straightforward to check that $[\sigma_n, \pi] = \alpha \sqcup(\beta + n_1) \in C(k_1, \ldots , k_m, \ell_1, \ldots , \ell_s),$ where the notation $\beta+n_1$ means to add $n_1$ to each element of $\beta.$
\end{proof}

\begin{remark}\label{rem:ParityDefectStability}
    We observe the following.
    \begin{enumerate}[label=\alph*), font=\normalfont]
        \item In the situation of Lemma \ref{lem:Bader1}, if $\pi$ has a parity defect, then $\pi_+$ has a parity defect as well.
        \item In the situation of Lemma \ref{lem:Bader2}, if $\pi_1$ has a parity defect, then $\pi$ has a parity defect as well.
    \end{enumerate}
\end{remark}

\subsection{Explicit permutations for irreducible conjugacy classes}
In this subsection, we use explicit $n$-cycles $\pi$ to show the statement of Proposition \ref{prop:type-A} for the following list of conjugacy classes, which we refer to as \textit{irreducible}:

$$\begin{array}{cc}
    C(n) & n \geqslant 5 \text{ odd,} \\[0.5ex]
    C(2k, 2l) & k, l \geqslant 1, \text{ except } k = l = 1, \\[0.5ex]
    C(n, 3) & n \geqslant 1 \text{ odd,} \\[0.5ex]
    C(2k, 2l, 3) & k, l \geqslant 1, \\[0.5ex]
    C(n, 2, 2) & n \geqslant 3 \text{ odd,} \\[0.5ex]
    C(2k, 2, 2, 2) & k \geqslant 1. \\[0.5ex]
\end{array}$$

The existence of these permutations follows from \cite{vavpetic} in the cases where the conjugacy class is contained in a symmetric group on an odd number of elements. For a conjugacy class in a symmetric group on an even number $n$ of elements, Vavpeti{\v{c}} constructs two $(n-1)$-cycles whose commutator is in the conjugacy class and remarks that the result can be extended to $n$-cycles as well.

For each of the irreducible conjugacy classes, we construct in Appendix~\ref{sec:appA} an $n$-cycle $\pi$  with $\pi(1) = 2$, which is parity unstable and which fulfills the assertion of Proposition~\ref{prop:type-A}, such that the commutator $[\sigma_n, \pi]$ lies in this class.

Finally, we are ready to prove Proposition \ref{prop:type-A}. The proof of the proposition uses Lemma \ref{lem:Bader1} and Lemma \ref{lem:Bader2}, and is similar to the induction arguments used in \cite[Corollary 2.5]{vavpetic} and \cite[Corollary 2.7]{vavpetic}.

\begin{proof}[Proof of Proposition \ref{prop:type-A}]
    In this proof, whenever we say \emph{a permutation for the conjugacy class} $C(l_1, \ldots , l_s),$ we mean an $m$-cycle $\pi,$ where $m := l_1+...+l_s,$ with the properties $\pi(1) = 2$ and $[\sigma_m, \pi] \in C(l_1, \ldots , l_s).$ Furthermore, we say that $C(l_1, \ldots , l_s)$ \emph{contains a number} $l$ if $l_j = l$ for some $j = 1, \ldots , s.$ 
    
    Let $C(k_1, \ldots, k_m)$ be an even conjugacy class. Our goal is to show that there is a parity unstable $\pi$ for $C(k_1, \ldots, k_m).$ 

    If $C(k_1, \ldots, k_m)$ does not contain the numbers $1, 2$ and $3,$ then we can find permutations for the irreducible conjugacy classes $C(k_i)$ with $k_i$ odd and $C(k_j, k_{j'})$ with $k_j, k_{j'}$ even. Then, by Lemma \ref{lem:Bader2} we can also find a $\pi$ for $C(k_1, \ldots , k_m).$

    So, we need to consider the cases where $C(k_1, \ldots, k_m)$ contains $1, 2$ or $3.$ The subtlety in these cases comes from the fact that there do not exist permutations for the conjugacy classes $C(3), C(2,2)$ and $C(2,2,1)$.

    Consider the permutations $(1 \ 2 \ 4 \ 3 \ 6 \ 5)$ and $(1 \ 2 \ 4 \ 9 \ 5 \ 7 \ 8 \ 6 \ 3).$ Their commutator with the corresponding shift permutations is contained in $C(3,3)$ and $C(3,3,3)$ respectively. So, by Lemma \ref{lem:Bader2}, we can find permutations for any conjugacy class of the form $C(3, \ldots, 3)$, that is, for any conjugacy class containing only and at least two $3$'s.

    Consider the permutations $(1 \ 2 \ 5 \ 4 \ 7 \ 8 \ 6 \ 3)$ and $(1 \ 2 \ 8 \ 12 \ 9 \ 6 \ 7 \ 3 \ 5 \ 10 \ 11 \ 4).$ Their commutator with the corresponding shift permutations is contained in $C(2,2,2,2)$ and $C(2,2,2,2,2,2)$ respectively. So, by Lemma \ref{lem:Bader2}, we can find permutations for any conjugacy class of the form $C(2, \ldots, 2),$ that is, for any conjugacy class containing only and at least two pairs of $2$'s.

    Consider the permutations $(1 \ 2 \ 4 \ 6 \ 7 \ 5 \ 3),$ $(1 \ 2 \ 4 \ 6 \ 7 \ 9 \ 10 \ 5 \ 8 \ 3)$ and $(1 \ 2 \ 7 \ 5 \ 3 \ 11 \ 10 \ 4 \ 6 \ 9 \ 8)$. The commutators of these three permutations with the corresponding shift permutations lie in $C(3,2,2), \, C(3,3,2,2)$ and $C(3,2,2,2,2)$, respectively. Using these permutations, together with the ones for the conjugacy classes $C(2, \ldots, 2), \, C(3, \ldots,3)$ and Lemma \ref{lem:Bader2}, we can construct permutations for every conjugacy class of the form $C(2, \ldots, 2, 3, \ldots,3)$ for any even number of $2$'s and any number of $3$'s.

    Finally, the commutator of the permutation $(1 \ 2 \ 6 \ 4 \ 5 \ 3)$ with $\sigma_6$ lies in $C(2, 2, 1, 1).$

    We go back to the general case of $C(k_1, \ldots, k_m)$. If $C(k_1, \ldots, k_m)$ contains $3$ and possibly $1$ but not $2$, then using permutations for $C(3, \ldots, 3)$ or the irreducible conjugacy classes, together with Lemmata \ref{lem:Bader1} and \ref{lem:Bader2} yields a permutation for $C(k_1, \ldots k_m)$.

    If $C(k_1, \ldots , k_m)$ contains $2$ and possibly $1$ but not $3$, then using permutations for $C(2, \ldots , 2)$, $C(2,2,1,1)$ and the irreducible conjugacy classes, together with Lemmata \ref{lem:Bader1} and \ref{lem:Bader2} yield a permutation for $C(k_1, \ldots , k_m)$.

    If $C(k_1, \ldots , k_m)$ contains $2$ and $3$ and possibly $1$, then using permutations for\linebreak $C(2, \ldots, 2, 3, \ldots, 3)$ and the irreducible conjugacy classes, together with Lemmata \ref{lem:Bader1} and \ref{lem:Bader2} yields a permutation for $C(k_1, \ldots , k_m)$.

    Finally, if $C(k_1, \ldots , k_m)$ contains $1$, but not $2$ and $3$, then using permutations for the irreducible conjugacy classes, together with Lemma \ref{lem:Bader1} yields a permutation for $C(k_1, \ldots, k_m)$.
    
    To conclude, note that all the permutations constructed for the irreducible conjugacy classes, as well as the ones constructed in this proof, are parity unstable. Together with Remark \ref{rem:ParityDefectStability}, this implies that the permutation constructed for $C(k_1, \ldots, k_m)$ is parity unstable.
\end{proof}

\section{The signed symmetric group}\label{section:type-B}

\begin{definition}
Let \(v\in\mathbb{F}_2^n\) and let
\(\pi\in\mathfrak{S}_n\). We put
\[
(\pi\mathbin{\cdot}v)_i:=v_{\pi^{-1}(i)}
\qquad (1\leq i\leq n).
\]
This defines a left action of \(\mathfrak{S}_n\) on
\(\mathbb{F}_2^n\), since
\[
(\pi\tau)\mathbin{\cdot}v
=
\pi\mathbin{\cdot}(\tau\mathbin{\cdot}v)
\qquad
(\pi,\tau\in\mathfrak{S}_n).
\]
\end{definition}

The Coxeter group \(W_{B_n}\) of type \(B_n\) can be realised as the
full symmetry group of the~$n$-dimensional hypercube $Q_n=[-1,1]^n$. Equivalently, it is the
group of signed permutations of~$n$ coordinates. Thus,
\[
W_{B_n}\cong
(\mathbb{Z}/2\mathbb{Z})^n\rtimes\mathfrak{S}_n,
\]
where the group~\(\mathfrak{S}_n\) acts by permuting
the coordinates, or equivalently, the factors of
\((\mathbb{Z}/2\mathbb{Z})^n\). Using the multiplicative realisation of the cyclic group of order two,
we may identify
\[
(\mathbb{Z}/2\mathbb{Z})^n
\cong
\left\{
(\varepsilon_1,\ldots,\varepsilon_n)
\ \middle|\
\varepsilon_i\in\{\pm1\},\ 1\leq i\leq n
\right\},
\]
where $\varepsilon_i$ records whether the $i$-th coordinate is left unchanged or sign-reversed.

For the computations below, it is convenient to identify the cyclic
group of order two with the additive group of the finite field
\(\mathbb{F}_2\). Accordingly, we write
\[
W_{B_n}\cong\mathbb{F}_2^n\rtimes\mathfrak{S}_n,
\]
and record that for
\(a,b\in\mathbb{F}_2^n\) and
\(\pi,\tau\in\mathfrak{S}_n\), the multiplication and inverse are given by
\[
(a,\pi)(b,\tau)
=
\bigl(a+\pi\mathbin{\cdot}b,\pi\tau\bigr)
\qquad\ \textup{and}\ \qquad\ 
(a,\pi)^{-1}
=
\bigl(\pi^{-1}\mathbin{\cdot}a,\pi^{-1}\bigr),
\]
respectively.\par
Since we want to show a result about commutators in $W_{B_n}$, we now explicitly calculate the commutator of two elements: let $x = (a,\pi), ~y = (b, \tau) \in W_{B_n}$. Then

\begin{equation}
\begin{aligned}
\relax [x,y]
={}&\Bigl(
 \pi^{-1}\mathbin{\cdot}a
 +(\pi^{-1}\tau^{-1})\mathbin{\cdot}b
 +(\pi^{-1}\tau^{-1})\mathbin{\cdot}a
 +(\pi^{-1}\tau^{-1}\pi)\mathbin{\cdot}b,
 [\pi,\tau]\Bigr)\\
={}&\Bigl(
 (\pi^{-1}\tau^{-1})\mathbin{\cdot}
 (a+\tau\mathbin{\cdot}a+b+\pi\mathbin{\cdot}b),
 [\pi,\tau]\Bigr).
\end{aligned}
\label{equ:commutatorBn}
\end{equation}
Indeed, a direct computation shows:
\[
\begin{aligned}
\relax [x,y]
&=(\pi^{-1}\mathbin{\cdot}a,\pi^{-1})
  (\tau^{-1}\mathbin{\cdot}b,\tau^{-1})(a,\pi)(b,\tau)\\
&=\bigl(\pi^{-1}\mathbin{\cdot}a
 +(\pi^{-1}\tau^{-1})\mathbin{\cdot}b,
 \pi^{-1}\tau^{-1}\bigr)(a,\pi)(b,\tau)\\
&=\bigl(\pi^{-1}\mathbin{\cdot}a
 +(\pi^{-1}\tau^{-1})\mathbin{\cdot}b
 +(\pi^{-1}\tau^{-1})\mathbin{\cdot}a,
 \pi^{-1}\tau^{-1}\pi\bigr)(b,\tau)\\
&=\Bigl(\pi^{-1}\mathbin{\cdot}a
 +(\pi^{-1}\tau^{-1})\mathbin{\cdot}b
 +(\pi^{-1}\tau^{-1})\mathbin{\cdot}a
 +(\pi^{-1}\tau^{-1}\pi)\mathbin{\cdot}b,
 [\pi,\tau]\Bigr)\\
&=\Bigl(
 (\pi^{-1}\tau^{-1})\mathbin{\cdot}
 (a+\tau\mathbin{\cdot}a+b+\pi\mathbin{\cdot}b),
 [\pi,\tau]\Bigr).
\end{aligned}
\]

For the following description of a Coxeter-generating set, we briefly
use the equivalent signed-permutation realisation of the group $W_{B_n}$. Namely, we identify
\(W_{B_n}\) with the group of all permutations~$w$ of
\[
\{\pm1,\ldots,\pm n\}
\]
satisfying
\[
w(-i)=-w(i)
\qquad (1\leq i\leq n).
\]
We use ordinary cycle notation on the signed
indices. Thus, \((i\, -i)\) represents a sign change in the \(i\)-th
coordinate, while
\[
(i\, i+1)(-i\, -(i+1))
\]
interchanges the \(i\)-th and \((1+i)\)-th coordinate directions. Then, a standard Coxeter-generating set for the group $W_{B_n}$ is
\begin{align*}
s_0&=(1\ \ -1),\\
s_i&=(i\ \ i+1)(-i\ \ -(i+1)),
\qquad 1\leq i\leq n-1.
\end{align*}
Equivalently, in the semidirect-product notation, if
\(e_1=(1,0,\ldots,0)\), then
\[
s_0=(e_1,\operatorname{id}),
\qquad
s_i=(0,(i\ \ i+1)).
\]
Hence, one of the Coxeter elements of the group $W_{B_n}$ is given by 
\begin{align} \label{equ:CoxElement}
c:= s_0 s_1 \cdots s_{n-1}
= (1\ \ 2 \ \ldots \ n\ \ -1 \ \ -2 \ \ldots \  -n). 
\end{align}

\begin{lemma}[{\cite[§3.16, Proposition]{Hum90}}]\label{lem:CoxEltsBnAreOneConjugacyClass}
In the group $W_{B_n}$, the set of all Coxeter elements forms precisely one conjugacy class.
\end{lemma}

We now briefly recall a standard fact about the set of all Coxeter elements of $W_{B_n}$ and show how this allows only certain entries in $\mathbb{F}_2^n$ for those elements.

\begin{proposition}
Let $x = (v,\pi) \in W_{B_n} = \mathbb{F}_2^n \rtimes \mathfrak{S}_n$. Then $x$ is a Coxeter element if and only if $\sum_{i=1}^n v_i = 1$ and $\pi$ is an $n$-cycle.
\end{proposition}

\begin{proof}
Considering the simple reflections $s_0, s_1, \ldots, s_{n-1}$ as elements of $\mathbb{F}_2^n \rtimes \mathfrak{S}_n$, we have
\begin{align*}
    s_0 & = (e_1, \idop)\\
    s_i & = (0, (i\ \ i+1)) \quad (1 \leq i \leq n-1)\\
    c & = s_0s_1 \cdots s_{n-1} = (e_1, (1,\ldots,n))
\end{align*}

A conjugacy class in $W_{B_n}$ is characterized by its signed cycle type (see \cite{Car72}). Considering $W_{B_n}$ again as the group $\mathbb{F}_2^n \rtimes \mathfrak{S}_n$, an element $x \in W_{B_n}$ is given as $(v,\pi)$, where $v \in \mathbb{F}_2^n$ and $\pi \in \mathfrak{S}_n$. If $\pi = \pi_1 \cdots \pi_r$ is the disjoint cycle decomposition, each of the cycles $\pi_i$ is equipped with a sign. The sign of $\pi_i$ is negative if $\displaystyle \sum_{j \in \mathrm{supp}(\pi_i)} v_j = 1$, and positive otherwise. Therefore, the conjugacy class of the Coxeter element $c = (e_n, (1,\ldots,n))$ consists precisely of those elements $(v, \pi)$, where $\pi$ is an $n$-cycle and $\displaystyle \sum_{i = 1}^n v_i = 1$. The claim follows now from Lemma \ref{lem:CoxEltsBnAreOneConjugacyClass}.
\end{proof}

\section{Sums of permuted vectors in $\mathbb{F}_2^n$} \label{sec:5}

\noindent For $\pi \in \mathfrak{S}_n$ we define the linear map
$$
P_{\pi}: \mathbb{F}_2^n \rightarrow \mathbb{F}_2^n, ~v \mapsto \pi\mathbin{\cdot}v.
$$
We further define the map $M_{\pi} := \idop + P_{\pi}$

\begin{definition}
For $1\leq i\leq n$ we denote by $e_i\in \mathbb{F}_2^n$ the $i$-th unit vector. For $1\leq k, \ell\leq n$ we set 
$$e_{k,\ell} := e_k + e_{\ell},$$
that is, the non-zero entries of $e_{k,\ell}$ are equal to one and occur precisely at positions $k$ and $\ell$. If $k = \ell$, the vector $e_{k, \ell}$ is the zero vector.
\end{definition}

\subsection{Balanced bipartitions}
\begin{definition}
Let $n$ be even and let $\pi = (\pi_1, \ldots, \pi_n) \in \mathfrak{S}_n$ be an $n$-cycle with $\pi_1=1$. We define two subsets
$$
R = R(\pi) := \{ \pi_1, \pi_3,  \ldots, \pi_{n-1} \} \quad \text{and} \quad B = B(\pi) := \{ \pi_2, \pi_4,  \ldots, \pi_{n} \}.
$$ 
We call $\mathcal{P} = \mathcal{P}(\pi) = \{ R,B\}$ the \textbf{balanced bipartition} (induced by $\pi$).
\end{definition}

\begin{definition}
\phantom{0}
    \begin{enumerate}[label=\alph*), font=\normalfont]
        \item We define the set $\mathbb{F}_{2, \mathrm{even}}^n := \left\{ x\in \mathbb{F}_2^n \mid \sum\limits_{i=1}^n x_i = 0\right\} $ and note that it has the structure of a subspace of $\mathbb{F}_2^n$.
        \item We define the set $\mathbb{F}_{2, \mathrm{odd}}^n := \left\{ x\in \mathbb{F}_2^n \mid \sum\limits_{i=1}^n x_i = 1\right\} $ and note that it has the structure of an affine hyperplane in $\mathbb{F}_2^n$.
        \item If $n$ is even and $\pi$ an $n$-cycle with balanced bipartition $\mathcal{P}(\pi) = \{R, B\}$, we define $U_\pi := \left\{ v \in \mathbb{F}_{2, \mathrm{even}}^n \mid \sum\limits_{i \in R} v_i = 0 \right\}$.
    \end{enumerate}
\end{definition}

Note that $U_{\pi}$ is well-defined since $\sum_{i\in R}v_i=0$ if and only if $\sum_{i\in B}v_i=0$.

\begin{lemma}\label{lem:ImageOfMpi}
Let $n$ be even and let $\pi$ be an $n$-cycle. If $\mathcal{P}(\pi) = \{R, B\}$, then
$$
\mathrm{im}\left( \left. M_{\pi}\right\vert_{\mathbb{F}_{2, \mathrm{even}}^n} \right) = U_\pi
$$
and this is a hyperplane in $\mathbb{F}_{2, \mathrm{even}}^n$.
\end{lemma}

\begin{proof}
First observe that every vector of the form $M_\pi v=v+\pi\mathbin{\cdot}v$ has even total coordinate sum, because
\[
   \sum_i(v_i+v_{\pi^{-1}(i)})
 =\sum_i v_i+\sum_i v_{\pi^{-1}(i)}
 =2\sum_i v_i=0
\]
in $\mathbb{F}_2$.  Hence $\textup{im}(M_\pi)\subseteq\mathbb{F}_{2,\mathrm{even}}^n$.

The kernel of $M_\pi$ consists of the vectors satisfying $v_i=v_{\pi^{-1}(i)}$ for all $i$.  Since $\pi$ is an $n$-cycle, this means all coordinates of $v$ are equal. Thus
\[
   \ker(M_\pi)=\{0,\mathbf{1}\},
\]
where $\mathbf{1}=(1,1,\ldots,1)$.

Recall that $n$ is even.  Write $\mathcal{P}(\pi)=\{R,B\}$.  The map $\pi$ swaps $R$ and $B$.  If $v\in\mathbb{F}_{2,\mathrm{even}}^n$, then $M_\pi v=v+P_\pi v$ has sum zero over $R$ due to the following argument. First note that
\[
 \sum_{i\in R}(M_\pi v)_i
 =\sum_{i\in R}v_i+\sum_{i\in R}v_{\pi^{-1}(i)}
 =\sum_{i\in R}v_i+\sum_{j\in B}v_j.
\]
Since $n$ is even, $\pi^{-1}$ sends $R$ bijectively onto
$B$. Hence, the second sum is $\sum\limits_{j\in B}^{} v_j$.  Since $v$ lies in~$\mathbb{F}_{2,\mathrm{even}}^n$, we obtain
\[
   \sum_{i\in R}v_i+\sum_{j\in B}v_j=\sum_{i=1}^n v_i=0.
\]
Thus,
$$
\mathrm{im}\left( \left. M_{\pi}\right\vert_{\mathbb{F}_{2, \mathrm{even}}^n} \right) \subseteq \left\{ v \in \mathbb{F}_{2, \mathrm{even}}^n \mid \sum_{i \in R} v_i = 0 \right\}.
$$
It remains to compare dimensions.  The vector $\mathbf{1}$ has even weight because $n$ is even, so $\mathbf{1}\in\mathbb{F}_{2,\mathrm{even}}^n$.  Therefore the kernel of $\left. M_{\pi}\right\vert_{\mathbb{F}_{2, \mathrm{even}}^n}$ has dimension $1$.  Hence, by the rank–nullity theorem, the image has dimension
\[
   (n-1)-1=n-2.
\]
The space $U_\pi$ is one linear equation inside $\mathbb{F}_{2,\mathrm{even}}^n$, so it also has dimension $n-2$.  Since the image is contained in $U_\pi$ and both have the same dimension, they are equal.
\end{proof}

\medskip
\begin{lemma} \label{lem:ImageMpiFornodd}
    Let $n$ be odd and $\pi$ an $n$-cycle. Then
    $\displaystyle \mathrm{im}\left( \left. M_{\pi}\right\vert_{\mathbb{F}_{2, \mathrm{even}}^n} \right) = \mathbb{F}_{2, \mathrm{even}}^n$.
\end{lemma}

\begin{proof}
As in the proof of Lemma \ref{lem:ImageOfMpi}, we consider the kernel of $\left. M_{\pi}\right\vert_{\mathbb{F}_{2, \mathrm{even}}^n}$. Since $\mathbf{1}\notin\mathbb{F}_{2,\mathrm{even}}^n$, this kernel is trivial. By the rank-nullity theorem the claim follows.
\end{proof}

\medskip
\begin{lemma}\label{lem:DistinctBipartitionsImplyDistinctHyperplanes}
Let $n$ be even and let $\pi$ and $\tau$ be $n$-cycles such that $\mathcal{P}(\pi) \neq \mathcal{P}(\tau)$. Then also $U_{\pi} \neq U_{\tau}$.
\end{lemma}

\begin{proof}
Write $\mathcal{P}(\pi) = \{R_1, B_1\}$ and $\mathcal{P}(\tau) = \{R_2, B_2\}$. By swapping $R_2$ and $B_2$ if necessary, we can assume, since $\mathcal{P}(\pi) \neq \mathcal{P}(\tau)$, that there are $1 \leq k, \ell \leq n$ with $k \in R_1$, $k, \ell \in R_2$ and $\ell \notin R_1$. But then $e_{k, \ell} \notin U_{\pi}$, while $e_{k, \ell} \in U_{\tau}$.
\end{proof}

\begin{proposition}\label{prop:BaakeYieldsVevenForneven}
Let $n \geq 6$ be even and $\pi, \tau \in \mathfrak{S}_n$ be $n$-cycles with $\mathcal{P}(\pi) \neq \mathcal{P}(\tau)$. Then
$$
 \{a+\tau\mathbin{\cdot}a+b+\pi\mathbin{\cdot}b
   \mid a,b\in\mathbb F_{2,\mathrm{odd}}^n\}
 =\mathbb F_{2,\mathrm{even}}^n.
$$
\end{proposition}

\begin{proof}
    Since $\mathcal P(\pi)\neq\mathcal P(\tau)$, it follows from Lemma \ref{lem:DistinctBipartitionsImplyDistinctHyperplanes} that $U_\pi \neq U_\tau$. By Lemma \ref{lem:ImageOfMpi}, $U_\pi$ and $U_\tau$ are two hyperplanes in $\mathbb{F}_{2,\mathrm{even}}^n$. Hence, their sum is $\mathbb{F}_{2,\mathrm{even}}^n$. As
    \begin{align*}
    \{a+\tau\mathbin{\cdot}a+b+\pi\mathbin{\cdot}b
      \mid a,b\in\mathbb F_{2,\mathrm{odd}}^n\} & = \{e_1+\tau\mathbin{\cdot}e_1+\widetilde a
       +\tau\mathbin{\cdot}\widetilde a
       +e_1+\pi\mathbin{\cdot}e_1+\widetilde b
       +\pi\mathbin{\cdot}\widetilde b
       \mid \widetilde a,\widetilde b\in
       \mathbb F_{2,\mathrm{even}}^n\} \\
    & = e_{\tau(1),\pi(1)}+U_\tau+U_\pi\\
    & = e_{\tau(1),\pi(1)}+\mathbb F_{2,\mathrm{even}}^n
     =\mathbb F_{2,\mathrm{even}}^n,
    \end{align*} 
    the claim follows.
\end{proof}

\begin{proposition}\label{prop:BaakeYieldsVevenFornodd}
Let $n > 6$ be odd and $\pi, \tau \in \mathfrak{S}_n$ be $n$-cycles. Then
$$
 \{a+\tau\mathbin{\cdot}a+b+\pi\mathbin{\cdot}b
   \mid a,b\in\mathbb F_{2,\mathrm{odd}}^n\}
 =\mathbb F_{2,\mathrm{even}}^n.
$$
\end{proposition}

\begin{proof}
By Lemma \ref{lem:ImageMpiFornodd} we have
    \begin{align*}
    \{a+\tau\mathbin{\cdot}a+b+\pi\mathbin{\cdot}b \mid a,b\in\mathbb F_{2,\mathrm{odd}}^n\}
    &=\{e_1+\tau\mathbin{\cdot}e_1+\widetilde a+\tau\mathbin{\cdot}\widetilde a
    +e_1+\pi\mathbin{\cdot}e_1+\widetilde b+\pi\mathbin{\cdot}\widetilde b
    \mid \widetilde a,\widetilde b\in\mathbb F_{2,\mathrm{even}}^n\}\\
    &=e_{\tau(1),\pi(1)}+\mathbb F_{2,\mathrm{even}}^n\\
    &=\mathbb F_{2,\mathrm{even}}^n.
    \end{align*}
\end{proof}

\subsection{Commuting $n$-cycles}

\begin{lemma} \label{lem:H1}
    Let $\pi$ be an $n$-cycle. Then $\pi^k$ is an $n$-cycle iff $\gcd(n,k) = 1$.
\end{lemma}

\begin{lemma} \label{lem:H2}
    Let $n$ be even. Let $\pi$ be an $n$-cycle and $k$ an integer with $\gcd(n,k) = 1$. Then $\pi^2$ and $(\pi^k)^2$ have the same orbits.
\end{lemma}

\begin{proof}
Let $\pi = (\pi_1, \ldots, \pi_n)$. The two orbits of $\pi^2$ are
$$
\mathcal{O}_1 = \{\pi_1, \pi_3, \ldots, \pi_{n-1}\} \quad \text{and} \quad \mathcal{O}_2= \{\pi_2, \pi_4, \ldots, \pi_{n}\}
$$
We obviously have $\pi^{2k}(\mathcal{O}_1) \subseteq \mathcal{O}_1$ and $\pi^{2k}(\mathcal{O}_2) \subseteq \mathcal{O}_2$. Since 
$$
\pi^2 = (\pi_1, \pi_3, \ldots, \pi_{n-1})(\pi_2, \pi_4, \ldots, \pi_{n}),
$$
the permutation $\pi^2$ acts as $(\pi_1, \pi_3, \ldots, \pi_{n-1})$ on the cycle $\mathcal{O}_1$ (and vice versa for $\mathcal{O}_2$). 
We also have $\gcd \left( \frac n2, k \right)= 1$. Therefore $\pi^{2k}$ has exactly one orbit when acting on $\mathcal{O}_1$ and when acting on $\mathcal{O}_2$, respectively. 
\end{proof}

\begin{lemma} \label{lem:H3}
Let $\pi, \tau \in \mathfrak{S}_n$ be two $n$-cycles. Then $[\pi, \tau] = \idop$ iff $\tau = \pi^k$ for some $k$ with $\gcd(k,n) = 1$.
\end{lemma}

\begin{proof}
    Assume $[\pi, \tau] = \idop$, hence $\tau$ lies in the centralizer of $\pi$, which is known to be the cyclic group generated by $\pi$. Therefore $\tau = \pi^k$ for some $k$ and we deduce from Lemma \ref{lem:H1} that $\gcd(k,n) =1$. The converse is immediate.
\end{proof}

\section{Proof of the second Main Theorem} \label{sec:ProofThmB}
\noindent For the standard cycle $\sigma_n=(1\ 2\ \cdots\ n)$, we define
\[
   \mathcal{P}_0 := \mathcal{P}(\sigma_n)=\bigl\{\{1,3,5,\ldots,n-1\},\{2,4,6,\ldots,n\}\bigr\}.
\]

\begin{lemma}\label{lem:alternating}
Let $n$ be even and let $\pi=(1\ 2\ i_3\cdots i_n)$ be normalised.  Then
\[
   \mathcal{P}(\pi)=\mathcal{P}_0
\]
if and only if $\pi$ is parity stable.
\end{lemma}

\begin{proof}
The two parts of $\mathcal{P}(\pi)$ are the symbols in odd positions and the symbols in even positions in the cyclic word for $\pi$:
\[
   \{1,i_3,i_5,\ldots,i_{n-1}\}
   \quad\text{and}\quad
   \{2,i_4,i_6,\ldots,i_n\}.
\]
The standard bipartition $\mathcal{P}_0$ consists of the odd numbers and the even numbers.

Since the first position already contains $1$ (which is odd), and the second position already contains $2$ (which is even), the two bipartitions agree exactly when all remaining odd positions contain odd numbers and all remaining even positions contain even numbers. This is precisely the definition of being parity stable.
\end{proof}

\begin{theorem} \label{thm:distinct-bipartition}
Let $n\geq 6$ be even. Every nonidentity element $\rho\in\mathfrak{A}_n$ can be written as
\[
   \rho=[\pi,\tau]
\]
with $\pi$ and $\tau$ both $n$-cycles and
\[
    \mathcal{P}(\pi)\neq \mathcal{P}(\tau).
\]
Moreover, if $[\pi,\tau]=1$ for two $n$-cycles, then
\[
    \mathcal{P}(\pi)= \mathcal{P}(\tau).
\]
\end{theorem}

\begin{proof}
First suppose $\rho =1$. By Lemma \ref{lem:H3} we have $\tau=\pi^k$ for some integer $k$ with $\gcd(k,n)=1$.  Since $n$ is even, $k$ must be odd.  The cycles $\pi^2$ and $\tau^2=\pi^{2k}$ have the same two orbits by Lemma \ref{lem:H2}. Therefore $\mathcal{P}(\pi)= \mathcal{P}(\tau)$.

Now let $\rho\neq 1$. By simultaneous conjugation, it is enough to realize one element in the conjugacy class of $\rho$ as $[\sigma_n,\tau]$ with $ \mathcal{P}(\tau)\neq  \mathcal{P}(\sigma_n)= \mathcal{P}_0$.
 
By Proposition~\ref{prop:type-A} we can choose the second $n$-cycle
\[
   \tau=(1\ 2\ i_3\ldots i_n)
\]
to be in normalised form and parity unstable. 
By Lemma~\ref{lem:alternating}, this means $\mathcal{P}(\tau)\neq \mathcal{P}_0$.

Finally, if our original element $\rho$ is not exactly this representative but only conjugate to it, choose $g\in\mathfrak{S}_n$ with $\rho=g[\sigma_n,\tau]g^{-1}$. Then
\[
   \rho=[g\sigma_ng^{-1},g\tau g^{-1}],
\]
and conjugation carries bipartitions to bipartitions:
\[
   \mathcal{P}(g\sigma_ng^{-1})=g\mathcal{P}_0,
   \qquad
   \mathcal{P}(g\tau g^{-1})=g\mathcal{P}(\tau).
\]
Since $\mathcal{P}(\tau)\neq \mathcal{P}_0$, the conjugated bipartitions are also distinct.  This proves the theorem.
\end{proof}

\begin{remark}
The theorem excludes the identity in its existence statement for a good reason.  If $[\pi,\tau]=1$, then the two cycles commute, and the first part of the proof shows that their bipartitions must be equal.  Thus the identity case must be handled separately in type $B_n$.
\end{remark}

\begin{definition}
We say that $v \in \mathbb{F}_{2,\mathrm{even}}^n$ is \textbf{attainable by} $\mathbf{\idop}$ if there exist $n$-cycles $\pi, \tau \in \mathfrak{S}_n$ and $a,b \in \mathbb{F}_{2,\mathrm{odd}}^n$ such that $[(a, \pi), (b, \tau)] = (v, \idop)$.
\end{definition}

\begin{lemma}\label{lem:Vectors_Attainable_By_Id}
    Let $n\geq 6$ be even. The vectors attainable by $\idop$ are precisely the elements of the union of all the hyperplanes $U_{\pi}$ for $\pi$ an $n$-cycle.
\end{lemma}

\begin{proof}
    Let $U$ be the set of vectors attainable by $\idop$. Let $\pi, \tau$ be $n$-cycles with $[\pi, \tau] = \idop$. By Lemma \ref{lem:H3}, $\tau = \pi^k$ for some $k$ with $\gcd(k,n) = 1$. Then
\begin{eqnarray*}
U' &:=& \{ [(a,\pi),(b,\pi^k)]\ |\ a,b\in\mathbb{F}_{2,\mathrm{odd}}^n\}\\
& \stackrel{(\ref{equ:commutatorBn})}{=} & \left\{ (\pi^{-1}\mathbin{\cdot}a
 +(\pi^{-(k+1)})\mathbin{\cdot}b
 +(\pi^{-(k+1)})\mathbin{\cdot}a
 +\pi^{-k}\mathbin{\cdot}b, \idop )\ |\ a,b\in\mathbb{F}_{2,\mathrm{odd}}^n \right\}\\
& = & \left\{ (P_{\pi^{-k-1}}(a+\pi^k\mathbin{\cdot}a+b+\pi\mathbin{\cdot}b), \idop) \ |\ a,b\in\mathbb{F}_{2,\mathrm{odd}}^n \right\}
\end{eqnarray*}
Let $\mathcal{P}(\pi) = \{R,B\}$. By Theorem \ref{thm:distinct-bipartition} we have $\mathcal{P}(\pi) = \mathcal{P}(\pi^k)$. As $k$ is odd, $\pi$ and $\pi^k$ swap $R$ and $B$. If $a \in \mathbb{F}_{2, \mathrm{odd}}^n$, then 
\begin{align} \label{equ:attainable1}
    \sum_{i \in R} (a+\pi^k\mathbin{\cdot}a)_i = \sum_{i \in R} a_i + \sum_{i \in R} a_{\pi^{-k}(i)} = \sum_{i \in R} a_i + \sum_{i \in B} a_{i} =  1
\end{align}
and, analogously we have
\begin{align} \label{equ:attainable2}
    \sum_{i\in B}(a+\pi^k\mathbin{\cdot}a)_i =\sum_{i\in B}a_i+\sum_{i\in B}a_{\pi^{-k}(i)}=1.
\end{align}
Equations (\ref{equ:attainable1}) and (\ref{equ:attainable2}) also hold when we sum over $b+\pi\mathbin{\cdot}b$ instead of $a+\pi^k \cdot a$. In particular, we obtain
\[
 \sum_{i\in R}
 (a+\pi^k\mathbin{\cdot}a+b+\pi\mathbin{\cdot}b)_i=0,
 \qquad
 \sum_{i\in B}
 (a+\pi^k\mathbin{\cdot}a+b+\pi\mathbin{\cdot}b)_i=0
\]
Since $-(k+1)$ is even, $P_{\pi^{-(k+1)}}$ preserves $R$ and $B$. Therefore, we see that 
$$
U'_{\mathrm{vec}} := \{u\in\mathbb F_2^n\mid
 [(a,\pi),(b,\pi^k)]=(u,\idop)
 \text{ for suitable }a,b\in\mathbb F_{2,\mathrm{odd}}^n\}
 \subseteq U_{\pi}
$$
that is, each vector attainable by $\idop$ is contained in some $U_{\pi}$ for some $n$-cycle $\pi$. 

It remains to show that for each $n$-cycle $\pi$, all vectors in $U_{\pi}$ are attainable by $\idop$. Therefore let $\pi$ be an $n$-cycle and note that for $a,b \in \mathbb{F}_{2, \mathrm{odd}}^n$ we have
\[
\begin{aligned}
\relax [(a,\pi),(b,\pi)]
&=\bigl(\pi^{-1}\mathbin{\cdot}a
 +\pi^{-2}\mathbin{\cdot}b
 +\pi^{-2}\mathbin{\cdot}a
 +\pi^{-1}\mathbin{\cdot}b,\idop\bigr)\\
&=\bigl(\pi^{-2}\mathbin{\cdot}M_\pi(a+b),\idop\bigr).
\end{aligned}
\]
The vectors attainable in this way are
\begin{align*}
\{ \pi^{-1}\mathbin{\cdot}a
 +\pi^{-2}\mathbin{\cdot}b
 +\pi^{-2}\mathbin{\cdot}a
 +\pi^{-1}\mathbin{\cdot}b \mid a,b \in \mathbb{F}_{2, \mathrm{odd}}^n\} & = \{ P_{\pi^{-1}}(a+b + \pi^{-1}\mathbin{\cdot}(a+b)) \mid a,b \in \mathbb{F}_{2, \mathrm{odd}}^n\} \\
& = \{ P_{\pi^{-1}}(c+\pi^{-1}\mathbin{\cdot}c) \mid c \in \mathbb{F}_{2, \mathrm{even}}^n\} \\
& = \{ c+\pi^{-1}\mathbin{\cdot}c \mid c \in \mathbb{F}_{2, \mathrm{even}}^n\} \\
& = \textup{im}\left(\left. M_{\pi^{-1}}\right\vert_{\mathbb{F}_{2, \mathrm{even}}^n}\right)\\
& = U_{\pi^{-1}},
\end{align*}
where the last equation is given by Lemma \ref{lem:ImageOfMpi}. Since $U_{\pi^{-1}} = U_{\pi}$, we see that each vector of $U_{\pi}$ is attainable by $\idop$.
\end{proof}

\begin{lemma}\label{lem:union-UP}
Let $n$ be even, let $v\in\mathbb{F}_{2,\mathrm{even}}^n$, and let $S:=\textup{supp}(v)=\{i\mid v_i=1\}.$ Then the following holds:
    \begin{enumerate}[label=\alph*), font=\normalfont]
        \item if $S\neq\{1,\ldots,n\}$, then $v\in U_\pi$ for some $n$-cycle $\pi$;
        \item if $S=\{1,\ldots,n\}$, then $v=\mathbf{1}$ lies in some $U_\pi$ for some $n$-cycle $\pi$ if and only if\linebreak $n\equiv 0\pmod 4$.
    \end{enumerate}
\end{lemma}

\begin{proof}
a) First suppose that $S\neq\{1,\ldots,n\}$.  Since $v\in\mathbb{F}_{2,\mathrm{even}}^n$, the number $|S|$ is also even. 
We want to choose an $n$-cycle $\pi$ with balanced bipartition $\mathcal{P} = \{R,B\}$ such that $|R\cap S|$ is even.  Then $v\in U_{\pi}$. 
Hence, it is enough to choose a set $R$ with the above properties. 

\begin{itemize}
    \item  If $|S|<n/2$, we may choose $R$ to contain all elements of $S$, and then fill the rest of $R$ using elements from the complement $S^c$. Then $|R\cap S|=|S|$, which is even.

    \item If $|S|\geq n/2$, there are two cases.  If $n/2$ is even, choose all $n/2$ elements of $R$ from $S$.  Then $|R\cap S|=n/2$ is even.  If $n/2$ is odd, choose $n/2-1$ elements of $R$ from $S$ and one element from $S^c$.  This is possible because $S\subsetneq \{1,\ldots , n\}$.  Then $|R\cap S|=n/2-1$, which is even.
\end{itemize}
Write $R=\{r_1,\ldots,r_m\}$ and $B=R^c=\{b_1,\ldots,b_m\}$ with $m=n/2$. Then $\pi=(r_1\ b_1\ r_2\ b_2\ \cdots\ r_m\ b_m)$ is an $n$-cycle with $\mathcal P(\pi)=\{R,B\}$.

\noindent b) Now suppose $S=\{1,\ldots,n\}$.  Then $v=\mathbf{1}$.  
For every $n$-cycle $\pi$ with balanced bipartition $\mathcal{P} = \{R,B\}$ we have
\[
   |R\cap S|=|R|=n/2.
\]
Thus $\mathbf{1}\in U_{\pi}$ if and only if $n/2$ is even.  This is equivalent to $n\equiv0\pmod4$.
\end{proof}

\begin{theorem}[Coxeter commutators in $B_n$]\label{thm:type-B}
Let $n\geq 6$ and let $W=W_{B_n}$.  Then
\[
\{ [c,d] \mid c,d \in W ~\mathrm{Coxeter~elements} \}=
\begin{cases}
   W'\setminus\{-\idop\}, & n\equiv 2 \pmod 4,\\
   W', & \text{else}.\\
\end{cases}
\]
Here, $-\idop=(\mathbf 1,\idop)$.
\end{theorem}

\begin{proof}
It is well-known that $W'=\mathbb{F}_{2,\mathrm{even}}^n\rtimes\mathfrak{A}_n$. We let
\[
   w=(v,\rho)\in W'.
\]

\medskip
\noindent\textbf{Case 1: $n$ is odd.} 
By Theorem~\ref{thm:type-A} we can choose $n$-cycles $\pi$ and $\tau$ such that
$$
[\pi, \tau] = \rho. 
$$
Then by Proposition \ref{prop:BaakeYieldsVevenFornodd}, it follows that the set
\begin{eqnarray*}
\{[(a,\pi),(b,\tau)] \mid a,b\in\mathbb F_{2,\mathrm{odd}}^n\}
& \stackrel{(\ref{equ:commutatorBn})}{=} & 
\{(P_{\pi^{-1}\tau^{-1}} (a+\tau\mathbin{\cdot}a+b+\pi\mathbin{\cdot}b),[\pi,\tau])
 \mid a,b\in\mathbb F_{2,\mathrm{odd}}^n\}\\
&=&\{(P_{\pi^{-1}\tau^{-1}}(c),\rho) \mid c\in\mathbb F_{2,\mathrm{even}}^n\}\\
&=&\{(c,\rho)\mid u\in\mathbb F_{2,\mathrm{even}}^n\}.
\end{eqnarray*}
contains $w$, where the last equality is due to the fact that $P_{\pi^{-1}\tau^{-1}}$ bijectively preserves $F_{2,\mathrm{even}}^n$.

\medskip
\noindent\textbf{Case 2: $n$ is even and $\rho\neq1$.}
By Theorem \ref{thm:distinct-bipartition}, choose $n$-cycles $\pi$ and $\tau$ such that
\[
   [\pi,\tau]=\rho
   \qquad\text{and}\qquad
   \mathcal{P}(\pi)\neq \mathcal{P}(\tau).
\]
Then by Proposition \ref{prop:BaakeYieldsVevenForneven}, it follows that the set 
\begin{eqnarray*}
\{[(a,\pi),(b,\tau)] \mid a,b\in\mathbb F_{2,\mathrm{odd}}^n\}
& \stackrel{(\ref{equ:commutatorBn})}{=} & 
\{(P_{\pi^{-1}\tau^{-1}} (a+\tau\mathbin{\cdot}a+b+\pi\mathbin{\cdot}b),[\pi,\tau])
 \mid a,b\in\mathbb F_{2,\mathrm{odd}}^n\}\\
&=&\{(P_{\pi^{-1}\tau^{-1}}(c),\rho) \mid c\in\mathbb F_{2,\mathrm{even}}^n\}\\
&=&\{(c,\rho)\mid c\in\mathbb F_{2,\mathrm{even}}^n\}.
\end{eqnarray*}
contains $w$, where the last equality is again due to the fact that $P_{\pi^{-1}\tau^{-1}}$ bijectively preserves $F_{2,\mathrm{even}}^n$.

\medskip
\noindent\textbf{Case 3: $n$ is even and $\rho=1$.}
By Lemma \ref{lem:Vectors_Attainable_By_Id}, the vectors attainable by $\idop$ are precisely the elements of the union of all the hyperplanes $U_{\pi}$ for $\pi$ an $n$-cycle. By Lemma~\ref{lem:union-UP}~a), every vector of $\mathbb{F}_{2,\mathrm{even}}^n\setminus\{\mathbf{1}\}$ lies in some $U_{\pi}$. By Lemma~\ref{lem:union-UP}~b), the vector $\mathbf{1}$ lies in some $U_{\pi}$ precisely when $n\equiv 0~\pmod{4}$.

\medskip
\noindent Combining the three cases proves the theorem.
\end{proof}

\appendix
\section{Explicit permutations for irreducible conjugacy classes} \label{sec:appA}
For us, the requirements a permutation $\pi$ needs to satisfy are being parity unstable and having $\pi(1) = 2$. 

In the following, whenever we write $\overset{+d}{\dots}$ or $\overset{-d}{\dots}$ for some $d \in \mathbb{N}$ in a cycle of a permutation, we mean to continue the cycle by adding or subtracting $d$ consecutively until reaching the next number in the permutation. For example $(1 \ 2 \ \overset{+2}{\dots} \ 8 \ 7) = (1 \ 2 \ 4 \ 6 \ 8 \ 7).$ Whenever we write $\overset{(+d, -d')}{\dots}$ for some $d,d' \in \mathbb{N},$ we mean to continue by adding $d,$ then subtracting $d'$ and so on until reaching the next number. For example $(1 \ 2 \ \overset{(+3, -1)}{\dots} \ 9) = (1 \ 2 \ 5 \ 4 \ 7 \ 6 \ 9).$

\smallskip
\fbox{$C(n)$ for $n \geqslant 5$ odd}\\

Let $\pi := (1 \ \ 2 \ \ 4 \ \ 5 \ \ \overset{+2}{\dots} \ \ n \ \ n-1 \ \ \overset{-2}{\dots} \ \ 6 \ \ 3).$ We compute:
$$[\sigma_n, \pi] = \begin{cases}
    (1 \ 4 \ 2 \ 3 \ 5), \text{ if } n = 5,\\ \\
    (1 \ 5 \ 2 \ 3 \ 7 \ 4 \ 6), \text{ if } n = 7,\\ \\
    (1 \ \ 5 \ \ \overset{+4}{\dots} \ \ n-4 \ \ n-1 \ \ \overset{-4}{\dots} \ \ 4 \ \ 7 \ \ \overset{+4}{\dots} \ \ n-2 \ \ 2 \ \ 3 \ \ n \ \ n-3 \ \ \overset{-4}{\dots} \ \ 6), \\[1ex] \text{if } n \equiv 1 \text{ mod } 4 \text{ and } n>5,\\ \\
    (1 \ \ 5 \ \ \overset{+4}{\dots} \ \ n-2 \ \ 2 \ \ 3 \ \ n \ \ n-3 \ \ \overset{-4}{\dots} \ \ 4 \ \ 7 \ \ \overset{+4}{\dots} \ \ n-4 \ \ n-1 \ \ \overset{-4}{\dots} \ \ 6), \\[1ex] \text{if } n \equiv 3 \text{ mod } 4 \text{ and } n > 7.
\end{cases}$$

In any case, we have $[\sigma_n, \pi] \in C(n).$\\

\fbox{$C(2k, 2l)$ for $k , l \geqslant 1$ except $k = l = 1$}\\

We subdivide this case into $k = l, k = l+1$ and $k > l + 1.$

For $k = l$ (and $l \geqslant 4$), let 
$$\pi = (1 \ \ 2 \ \ 5 \ \ 3 \ \ 4 \ \ \overset{(+3,-1)}{\dots} \ \ 2l+2 \ \  2l+6 \ \ 2l+4 \ \ 2l+5 \ \ \overset{(+3,-1)}{\dots} \ \ 4l \ \ 4l-1).$$
We compute:

$$[\sigma_{4l}, \pi] = \begin{cases}
    (1 \ \ 4 \ \ \overset{+4}{\dots} \ \ 2l \ \ 2l+5 \ \ \overset{+4}{\dots} \ \ 4l-3 \ \ 4l-2 \ \ \overset{-4}{\dots} \ \ 2l+6 \ \ 2l+3 \ \ \overset{-4}{\dots} \ \ 3)\\
    (2 \ \ \overset{+4}{\dots} \ \ 2l+2 \ \ 2l+7 \ \ \overset{+4}{\dots} \ \ 4l-1 \ \ 4l \ \ \overset{-4}{\dots} \ \ 2l+4 \ \ 2l+1 \ \ \overset{-4}{\dots} \ \ 5), \\[1ex] \text{if } l \equiv 0 \text{ mod } 2,\\ \\
    (1 \ \ 4 \ \ \overset{+4}{\dots} \ \ 2l+2 \ \ 2l+7 \ \ \overset{+4}{\dots} \ \ 4l-3 \ \ 4l-2 \ \ \overset{-4}{\dots} \ \ 2l+4 \ \ 2l+1 \ \ \overset{-4}{\dots} \ \ 3)\\
    (2 \ \ \overset{+4}{\dots} \ \ 2l \ \ 2l+5 \ \ \overset{+4}{\dots} \ \ 4l-1 \ \ 4l \ \ \overset{-4}{\dots} \ \ 2l+6 \ \ 2l+3 \ \ \overset{-4}{\dots} \ \ 5), \\[1ex] \text{if } l \equiv 1 \text{ mod } 2.
\end{cases}$$

Counting the elements in the two cycles of $[\sigma_{4l}, \pi]$ shows that in both cases $[\sigma_{4l}, \pi] \in C(2l,2l) = C(2k,2l).$ For $l \in \{2, 3\},$ the respective permutations $(1 \ 2 \ 4 \ 5 \ 8 \ 6 \ 7 \ 3)$ and $(1 \ 2 \ 5 \ 3 \ 4 \ 7 \ 6 \ 9 \ 8 \ 12 \ 10 \ 11)$ give the desired results $C(4,4)$ and $C(6,6)$.

For $k = l+1$ (and $l \geqslant 4$) let 
$$\pi = (1 \ \ 2 \ \ 5 \ \ 3 \ \ 4 \ \ \overset{(+3,-1)}{\dots} \ \ 2l+2 \ \ 2l+6 \ \ 2l+4 \ \ 2l+5 \ \ \overset{(+3,-1)}{\dots} \ \ 4l \ \ 4l-1 \ \ 4l+2 \ \ 4l+1).$$
Since this permutation is obtained by the one in the $k=l$ case by adding $4l+2, 4l+1$ at the end, which doesn't change the $(+3,-1)$ pattern, a similar computation to the one before shows $[\sigma_{4l+2}, \pi] \in C(2l + 2, 2l) = C(2k, 2l).$ For $l \in \{1, 2, 3\},$ the respective permutations $(1 \ 2 \ 5 \ 4 \ 6 \ 3) ,(1 \ 2 \ 5 \ 3 \ 4 \ 7 \ 6 \ 10 \ 8 \ 9)$ and $(1 \ 2 \ 5 \ 3 \ 4 \ 7 \ 6 \ 9 \ 8 \ 12 \ 10 \ 11 \ 14 \ 13)$ give the desired results $C(4,2), C(6,4)$ and $C(8,6).$

For $k > l+1$ (and $l \geqslant 2$), let 
$$\pi = (1 \ \ 2 \ \ 5 \ \ 3 \ \ 4 \ \ \overset{(+3,-1)}{\dots} \ \ 4l-4 \ \ 4l \ \ 4l-2 \ \ 4l-1 \ \ 4l+1 \ \ 4l+4 \ \ 4l+2 \ \ 4l+3 \ \ \overset{(+3,-1)}{\dots} \ \ 2k+2l \ \ 2k+2l-1).$$
We compute:

$$[\sigma_{2k+2l}, \pi] = \begin{cases}
    (1 \ \ 4 \ \ \overset{+4}{\dots} \ \ 4l-4 \ \ 4l-2 \ \ 4l-5 \ \ \overset{-4}{\dots} \ \ 3)\\
    (2 \ \ \overset{+4}{\dots} \ \ 4l-6 \ \ 4l-1 \ \ \overset{+4}{\dots} \ \ 2k+2l-1 \ \ 2k+2l \ \ \overset{-4}{\dots} \ \ 4l+4 \\ 4l+1 \ \ \overset{+4}{\dots} \ \ 2k+2l-3 \ \ 2k+2l-2 \ \ \overset{-4}{\dots} \ \ 4l+2 \ \ 4l \ \ 4l-3 \ \ \overset{-4}{\dots} \ \ 5), \\[1ex]  \text{if } k+l \equiv 0 \text{ mod } 2, \\ \\
    (1 \ \ 4 \ \ \overset{+4}{\dots} \ \ 4l-4 \ \ 4l-2 \ \ 4l-5 \ \ \overset{-4}{\dots} \ \ 3)\\
    (2 \ \ \overset{+4}{\dots} \ \ 4l-6 \ \ 4l-1 \ \ \overset{+4}{\dots} \ \ 2k+2l-3 \ \ 2k+2l-2 \ \ \overset{-4}{\dots} \ \ 4l+4 \\ 4l+1 \ \ \overset{+4}{\dots} \ \ 2k+2l-1 \ \ 2k+2l \ \ \overset{-4}{\dots} \ \ 4l+2 \ \ 4l \ \ 4l-3 \ \ \overset{-4}{\dots} \ \ 5),\\[1ex] \text{if } k+l \equiv 1 \text{ mod } 2.
\end{cases}$$

Again counting the elements in the two cycles of $[\sigma_{2k+2l}, \pi]$ shows that in both cases $[\sigma_{2k+2l}, \pi] \in C(2k,2l).$ The remaining case is $k > l+1$ and $l=1.$ If we let $n := 2k+2,$ this corresponds to $C(n-2,2)$ for $n \geqslant 8.$ For this (if $n \geqslant 10$), let 
$$\pi = (1 \ \ 2 \ \ 4 \ \ 5 \ \ 8 \ \ 7 \ \ 10 \ \ \overset{+2}{\dots} \ \ n \ \ n-1 \ \ \overset{-2}{\dots} \ \ 9 \ \ 6 \ \ 3),$$
and compute:
$$[\sigma_n, \pi] = \begin{cases}
    (4 \ 8) \\ (1 \ \ 5 \ \ 10 \ \ \overset{+4}{\dots} \ \ n-2 \ \ 2 \ \ 3 \ \ n \ \ n-3 \ \ \overset{-4}{\dots} \ \ 9 \ \ 7 \ \ 12 \ \ \overset{+4}{\dots} \ \ n-4 \ \ n-1 \ \ \overset{-4}{\dots} \ \ 11 \ \ 6), \\[1ex] \text{if } n \equiv 0 \text{ mod } 4,\\ \\
    (4 \ 8) \\ (1 \ \ 5 \ \ 10 \ \ \overset{+4}{\dots} \ \ n-4 \ \ n-1 \ \ \overset{-4}{\dots} \ \ 9 \ \ 7 \ \ 12 \ \ \overset{+4}{\dots} \ \ n-2 \ \ 2 \ \ 3 \ \ n \ \ n-3 \ \ \overset{-4}{\dots} \ \ 11 \ \ 6), \\[1ex] \text{if } n \equiv 2 \text{ mod } 4.
\end{cases}$$

Hence, $[\sigma_n, \pi] \in C(n-2,2).$ Finally, for $n = 8$, the permutation $(1 \ 2 \ 4 \ 6 \ 8 \ 7 \ 5 \ 3)$ gives the desired conjugacy class $C(6,2).$\\

\fbox{$C(n, 3)$ for $n \geqslant 1$ odd}\\

For $n = 1,$ $\pi = (1 \ 2 \ 4 \ 3)$ and for $n = 3,$ $\pi = (1 \ 2 \ 4 \ 3 \ 6 \ 5)$ give the desired result. For $n \geqslant 5,$ let $\pi = (1 \ \ 2 \ \ 4 \ \ 5 \ \ 7 \ \ 8 \ \ \overset{+2}{\dots} \ \ n + 3 \ \ n + 2 \ \ \overset{-2}{\dots} \ \ 9 \ 6 \ 3).$ We compute:

$$[\sigma_{n+3}, \pi] =  \begin{cases}
    (1 \ 5 \ 6) \, (2 \ \ 3 \ \ n+3 \ \ n \ \ \overset{-4}{\dots} \ \ 9 \ \ 4 \ \ \overset{+4}{\dots} \ \ n-1 \ \ n+2 \ \ \overset{-4}{\dots} \ \ 7 \ \ 10 \ \ \overset{+4}{\dots} \ \ n+1) \\[1ex] \text{if } n \equiv 1 \text{ mod } 4,\\ \\
    (1 \ 5 \ 6) \, (2 \ \ 3 \ \ n+3 \ \ n \ \ \overset{-4}{\dots} \ \ 7 \ \ 10 \ \ \overset{+4}{\dots} \ \ n-1 \ \ n+2 \ \ \overset{-4}{\dots} \ \ 9 \ \ 4 \ \ \overset{+4}{\dots} n+1) \\[1ex] \text{if } n \equiv 3 \text{ mod } 4.
\end{cases}$$

In any case, we have $[\sigma_{n+3}, \pi] \in C(n,3).$\\

\fbox{$C(2k, 2l, 3)$ for $k, l \geqslant 1$}\\

Let $n := 2k+2l$ and $\pi \in C(n)$ be any of the permutations constructed in the $C(2k, 2l)$ case above such that $[\sigma_n, \pi] \in C(2k,2l).$ Then $\pi$ is of the form $(1 \ 2 \ ... \ i \ n \ ... \ j),$ where $i, j \in \{3, \ldots , n-1\}.$ Define $\tau \in C(n+5)$ by 
$$\tau := (1 \ \ 2 \ \ ... \ \ i \ \ n \ \ ... \ \ j \ \ n+1 \ \ n+3 \ \ n+2 \ \ n+5 \ \ n+4),$$

where $\tau$ is the same as $\pi$ from $1$ to $j.$ Note that whatever $\pi$ is, the commutator $[\sigma_{n+5}, \tau]$ contains the cycle $(n+1 \ \ n+4 \ \ n+5).$ For any $m \in \{1, \ldots , n\} \, \backslash \, \{i, j\}$ it is straightforward that $[\sigma_{n+5}, \tau](m) = [\sigma_{n}, \pi](m).$ Furthermore, we also have $[\sigma_{n+5}, \tau](i) = j-1 = [\sigma_{n}, \pi](i).$ For $j$ we have $[\sigma_{n}, \pi](j) = n,$ whereas the commutator $[\sigma_{n+5}, \tau]$ contains a cycle of the form $(... \ \ j \ \ n+2 \ \ n+3 \ \ n \ \ ...).$ Since $[\sigma_{n}, \pi] \in C(2k, 2l)$ and one can check that the element $n$ is contained in the $2k$ cycle of $[\sigma_{n}, \pi]$ for any permutation $\pi$ constructed above, we conclude that $[\sigma_{n+5}, \tau] \in C(2k+2, 2l, 3).$

Using the permutations for the $C(2k,2l)$ case above and the method just described, we cover every case except for $C(4, 2, 3)$ and $C(2k, 2k, 3)$ for $k \geqslant 1.$ The permutation $(1 \ 2 \ 4 \ 5 \ 8 \ 7 \ 9 \ 6 \ 3)$ gives the desired result for $C(4, 2, 3).$ For $C(2k, 2k , 3)$ (and $k \geqslant 4$) consider 
$$\pi = (1 \ \ 2 \ \ 6 \ \ 3 \ \ 4 \ \ 10 \ \ 5 \ \ 8 \ \ 9 \ \ 7 \ \ 12 \ \ \overset{(-1, +3)}{\dots} \ \ 4k-3 \ \ 4k-1 \ \ 4k+1 \ \ 4k \ \ 4k+3 \ \ 4k+2).$$  
We compute

\begin{align*}
    [\sigma_{4k+3}, \pi] = & (1 \ \ 5 \ \ 7 \ \ 13 \ \ \overset{+4}{\dots} \ \ 4k-3 \ \ 4k \ \ 4k+1 \ \ 4k-2 \ \ \overset{-4}{\dots} \ \ 10)\\
    & (2 \ \ 8 \ \ 3 \ \ 9 \ \ 4 \ \ 11 \ \ \overset{+4}{\dots} \ \ 4k-5 \ \ 4k-4 \ \ \overset{-4}{\dots} \ \ 12 \ \ 6)
    (4k-1 \ \ 4k+2 \ \ 4k+3),
\end{align*}

and see that $[\sigma_{4k+3}, \pi] \in C(2k, 2k, 3).$

Finally, the permutations $(1 \ 2 \ 4 \ 6 \ 7 \ 5 \ 3)$ and $(1 \ 2 \ 4 \ 6 \ 5 \ 3 \ 7 \ 9 \ 8 \ 11 \ 10)$ as well as $(1 \ 2 \ 6 \ 3 \ 4 \ 10 \ 5 \ 8 \ 9 \ 7 \ 11 \ 13 \ 12 \ 15 \ 14)$ give the desired result for the conjugacy classes $C(2, 2, 3), C(4, 4, 3)$ and $C(6, 6, 3)$ respectively.\\

\fbox{$C(n, 2, 2)$ for $n \geqslant 3$ odd}\\

For $n =3,$ $\pi = (1 \ 2 \ 4 \ 6 \ 7 \ 5 \ 3),$ and for $n = 5,$ $\pi = (1 \ 2 \ 4 \ 7 \ 9 \ 8 \ 5 \ 6 \ 3)$ gives the desired result. For $n \geqslant 7,$ we let $$\pi = (1 \ \ 2 \ \ 4 \ \ n+2 \ \ n+4 \ \ n+3 \ \ 5 \ \ 6 \ \ \overset{+2}{\dots} \ \ n+1 \ \ n \ \ \overset{-2}{\dots} \ \ 7 \ \ 3).$$ We compute:
$$[\sigma_{n+4}, \pi] = \begin{cases}
    (2 \ n+2) \, (4 \ n+3) \\ (1 \ \ 6 \ \ \overset{+4}{\dots} \ \ n-3 \ \ n \ \ \overset{-4}{\dots} \ \ 5 \ \ 8 \ \ \overset{+4}{\dots} \ \ n-1 \ \ 3 \ \ n+4 \ \ n+1 \ \ n-2 \ \ \overset{-4}{\dots} \ \ 7) \\[1ex] \text{if } n \equiv 1 \text{ mod } 4,\\ \\
    (2 \ n+2) \, (4 \ n+3) \\ (1 \ \ 6 \ \ \overset{+4}{\dots} \ \ n-1 \ \ 3 \ \ n+4 \ \ n+1 \ \ n-2 \ \ \overset{-4}{\dots} \ \ 5 \ \ 8 \ \ \overset{+4}{\dots} \ \ n-3 \ \ n \ \ \overset{-4}{\dots} \ \ 7) \\[1ex] \text{if } n \equiv 3 \text{ mod } 4.
\end{cases}$$

In any case, we have $[\sigma_{n+4}, \pi] \in C(n, 2, 2).$\\

\fbox{$C(2k, 2, 2, 2)$ for $k \geqslant 1$}\\

For $k = 1,$ $\pi = ( 1 \ 2 \ 5 \ 4 \ 7 \ 8 \ 6 \ 3)$ and for $k = 2,$ $\pi = (1 \ 2 \ 4 \ 8 \ 9 \ 7 \ 5 \ 6 \ 10 \ 3)$ give the desired result. For $k \geqslant 3,$ we distinguish whether $k \equiv 0 \ (\text{mod } 2)$ or $k \equiv 1 \ (\text{mod } 2).$ In the first case, let $\pi = (1 \ \ 2 \ \ 5 \ \ 4 \ \ 7 \ \ 9 \ \ 10 \ \ \overset{+2}{\dots} \ \ n-2 \ \ n-3 \ \ \overset{-2}{\dots} \ \ 11 \ \ 8 \ \ n-1 \ \ n \ \ 6 \ \ 3).$ In the second case, let $\pi = (1 \ \ 2 \ \ 5 \ \ 4 \ \ 7 \ \ \overset{+2}{\dots} \ \ n-3 \ \ n-2 \ \ \overset{-2}{\dots} \ \ 8 \ \ n-1 \ \ n \ \ 6 \ \ 3).$ Here, $n := 2k + 6.$ 

We compute:

$$[\sigma_n, \pi] = \begin{cases}
    (1 \ 5) \, (2 \ n-1) \, (3 \ n) \\ (4 \ \ 10 \ \ \overset{+4}{\dots} \ \ n-4 \ \ 7 \ \ 8 \ \ n-2 \ \ n-5 \ \ \overset{-4}{\dots} \ \ 9 \ \ 12 \ \ \overset{+4}{\dots} \ \ n-6 \ \ n-3 \ \ \overset{-4}{\dots} \ \ 11 \ \ 6) \\[1ex] \text{if } k \equiv 0 \text{ mod } 2,\\ \\
    (1 \ 5) \, (2 \ n-1) \, (3 \ n) \\ (4 \ \ 9 \ \ \overset{+4}{\dots} \ \ n-3 \ \ 7 \ \ \overset{+4}{\dots} \ \ n-5 \ \ n-4 \ \ \overset{-4}{\dots} \ \ 8 \ \ n-2 \ \ \overset{-4}{\dots} \ \ 6) \\[1ex] \text{if } k \equiv 1 \text{ mod } 2.
\end{cases}$$

In any case, we have $[\sigma_n, \pi] \in C(2k, 2, 2, 2).$

\bibliography{commutator}
\bibliographystyle{amsplain}
\end{document}